\documentclass[]{interact}

\usepackage{epstopdf}
\usepackage{caption}
\usepackage{cases}
\usepackage{subfigure}
\usepackage{graphics}

\usepackage[numbers,sort&compress]{natbib}
\bibpunct[, ]{[}{]}{,}{n}{,}{,}
\makeatletter
\def\NAT@def@citea{\def\@citea{\NAT@separator}}
\makeatother

\theoremstyle{plain}
\newtheorem{theorem}{Theorem}[section]
\newtheorem{lemma}[theorem]{Lemma}

\newtheorem{alg}{Algorithm}
\theoremstyle{definition}

\newtheorem{example}[theorem]{Example}

\theoremstyle{remark}
\newtheorem{remark}{Remark}

\begin{document}


\title{Further study on two fixed point iterative schemes for absolute value equations}

\author{
\name{Jiayu Liu\textsuperscript{a}\thanks{Email address: 1977078576@qq.com.} and Tingting Luo\textsuperscript{a}\thanks{Email address: 610592494@qq.com.} and Cairong Chen\textsuperscript{b}\thanks{Corresponding author. Email address: cairongchen@fjnu.edu.cn.}}
\affil{\textsuperscript{a}School of Mathematics and Statistics, Fujian Normal University, Fuzhou,
350117, P.R. China; \textsuperscript{b}School of Mathematics and Statistics \& Key Laboratory of Analytical Mathematics and Applications (Ministry of Education) \& Fujian Provincial Key Laboratory of Statistics and Artificial Intelligence, Fujian Normal University, Fuzhou, 350117, P.R. China}
}

\maketitle

\begin{abstract}
In this paper, we reconsider two new iterative methods for solving absolute value equations (AVE), which is proposed by Ali and Pan (Jpn. J. Ind. Appl. Math. 40: 303--314, 2023). Convergence results of the two iterative schemes and new sufficient conditions for the unique solvability of AVE are presented. In addition, for a special case, the optimal iteration parameters of the two algorithms are analyzed, respectively. Numerical results demonstrate our claims.
\end{abstract}

\begin{keywords}
Absolute value equations; unique solvability; iterative method; convergence; optimal iteration parameter
\end{keywords}

\section{Introduction}

We consider absolute value equations (AVE) of the type
\begin{equation}\label{eq:ave}
Ax - \vert x \vert = b,
\end{equation}
where~$A\in\mathbb{R}^{n\times n}$ and $b\in\mathbb{R}^n$ are known, and~$x\in\mathbb{R}^n$ is unknown. Here, $\vert x\vert=(\vert x_1\vert,\vert x_2\vert,\cdots,\vert x_n\vert)^\top$. Obviously, AVE~\eqref{eq:ave} is a special case of the following generalized absolute value equations (GAVE)
\begin{equation}\label{eq:gave}
Ax + B\vert x \vert = b,
\end{equation}
where~$B \in \mathbb{R}^{n\times n}$ is given. To our knowledge, GAVE~\eqref{eq:gave} was first introduced in \cite{rohn2004}. As shown in \cite{mang2007}, solving the general GAVE~\eqref{eq:gave} is NP-hard. Moreover, if solvable, checking whether GAVE~\eqref{eq:gave} has a unique solution or multiple solutions is NP-complete \cite{prok2009}.

The significance of the AVE~\eqref{eq:ave} mainly arises from its equivalence with the linear complementarity problem (LCP) \cite{mame2006,huhu2010}. In addition, AVE~\eqref{eq:ave} also appears in the characterization of solutions of interval systems of linear equations \cite{rohn1989} or in the characterization of the regularity of interval matrices \cite{rohn2009b}. Over the past two decades, AVE~\eqref{eq:ave} have received considerable attention in the field of numerical optimization and numerical algebra. There are already many results on both the theoretical side and the numerical side of AVE~\eqref{eq:ave}.

On the theoretical side, conditions for the existence,  nonexistence and uniqueness of solutions to AVE~\eqref{eq:ave} are reported; see, e.g.,  \cite{mame2006,rohn2009,wuli2020,hlmo2023,rohf2014} and references therein. Relationships between AVE~\eqref{eq:ave} and LCP are investigated in \cite{prok2009,huhu2010,mang2014,mame2006}. More properties of the solution to AVE~\eqref{eq:ave} can be found in \cite{hlad2023,hlad2018,kemo2012,zahl2023}, only a few are listed. On the numerical side, a lot of algorithms have been developed to solve AVE~\eqref{eq:ave}. For instance, the Newton-type iteration methods \cite{mang2009,tzcy2023,crfp2016,zhwe2009,edhs2017}, the Picard-like iteration methods \cite{salk2014,rohf2014}, the matrix splitting iteration methods \cite{soso2023,doss2020,kema2017,guwl2019,chyh2023,zhzl2023}, the concave minimization approaches \cite{mang2007b,zahl2021}, the dynamic models \cite{maer2018,cyyh2021,yzch2024} and so on; see, e.g.,  \cite{chyh2023b,alct2023} and references therein.

Recently, Ali and Pan~\cite{alpa2023} proposed two new iterative schemes for solving AVE~\eqref{eq:ave} with $A$ being an $M$-matrix\footnote{In this paper, we do not assume that $A$ is an $M$-matrix since it is not necessary in the theoretical analysis.}. Concretely, let the matrix $A$ be  split into
\begin{equation*}
A = N_A - M_A,
\end{equation*}
where
\begin{equation*}
N_A=D_A-U_A+U^\top_A\quad \text{and}\quad M_A=L_A+U^\top_A,
\end{equation*}
in which $D_A, -L_A, -U_A$ are the diagonal, strictly lower triangular and strictly upper triangular parts of~$A$, respectively. Then the authors in \cite{alpa2023} first introduced the following two iterative schemes (see Algorithm~\ref{alg:s1} and  Algorithm~\ref{alg:s2}) for solving AVE~\eqref{eq:ave}.

\begin{alg}\label{alg:s1}
Let $A\in \mathbb{R}^{n\times n}$ and $b\in \mathbb{R}^{n}$. Given an initial vector $x^{(0)}\in \mathbb{R}^{n}$, for $k=0,1,2,\cdots$ until the iteration sequence $\left\{x^{(k)}\right\}_{k=0}^\infty$ is convergent, compute
\begin{equation*}
x^{(k+1)}=x^{(k)} - \lambda E[-M_A x^{(k+1)} + N_A x^{(k)} - (|x^{(k)}|+b)],
\end{equation*}
where $\lambda>0$ is a given parameter and $E\in\mathbb{R}^{n\times n}$ is a given positive diagonal matrix.
\end{alg}

\begin{alg}\label{alg:s2}
Let $A\in \mathbb{R}^{n\times n}$ and $b\in \mathbb{R}^{n}$. Given an initial vector $x^{(0)}\in \mathbb{R}^{n}$, for $k=0,1,2,\cdots$ until the iteration sequence $\left\{x^{(k)}\right\}_{k=0}^\infty$ is convergent, compute
\begin{equation*}
x^{(k+1)} = x^{(k)} + D_A^{-1} M_A x^{(k+1)} - \lambda E\left[A x^{(k)} - |x^{(k)}|-b\right] - D_A^{-1} M_A x^{(k)},
\end{equation*}
where $\lambda>0$ is a given parameter and $E\in\mathbb{R}^{n\times n}$ is a given positive diagonal matrix.
\end{alg}

The authors in \cite{alpa2023} also analyzed the convergence of Algorithm~\ref{alg:s1} and Algorithm~\ref{alg:s2}. (See the following Theorem~\ref{thm:alg1} and Theorem~\ref{thm:alg2}).

\begin{theorem}\label{thm:alg1}
Let $A = N_A - M_A$ be non-singular, $\lambda > 0$ and $E\in \mathbb{R}^{n\times n}$ be a diagonal matrix with positive diagonal elements.  If
\begin{equation}\label{ie:condition1}
\rho(R^{-1}S)<1
\end{equation}
with $R = I - \lambda E |M_A|$ and $S = \lambda E + |I - \lambda EN_A|$, then for any initial point $x^{(0)}\in \mathbb{R}^n$, the sequence $\{x^{(k)}\}_{k=0}^\infty$ generated by Algorithm~\ref{alg:s1} converges to the unique solution of AVE~\eqref{eq:ave}.
\end{theorem}

\begin{theorem}\label{thm:alg2}
Let $A = N_A - M_A$ be non-singular, $\lambda > 0$ and $E\in \mathbb{R}^{n\times n}$ be a diagonal matrix with positive diagonal elements.  If
\begin{equation}\label{ie:condition2}
\rho(G^{-1}J)<1
 \end{equation}
with $G = I - D_A^{-1} |M_A|$ and $J = I + \lambda E - |\lambda EA + D_A^{-1}M_A|$, then for any initial point $x^{(0)}\in \mathbb{R}^n$, the sequence $\{x^{(k)}\}_{k=0}^\infty$ generated by Algorithm~\ref{alg:s2} converges to the unique solution of AVE~\eqref{eq:ave}.
\end{theorem}

However, in the numerical experiments of \cite{alpa2023}, Algorithm~\ref{alg:s1} and  Algorithm~\ref{alg:s2} did not be tested. As an alternative, with the idea of predictor-corrector, the authors  in \cite{alpa2023} executed Algorithm~\ref{alg:s1} and  Algorithm~\ref{alg:s2} by the following Algorithm~\ref{alg:s3} and  Algorithm~\ref{alg:s4}, respectively. Numerical results of \cite{alpa2023} demonstrate that Algorithm~\ref{alg:s3} and  Algorithm~\ref{alg:s4} perform well. By further investigation, we find that Algorithm~\ref{alg:s1} and Algorithm~\ref{alg:s3} are fundamentally different, and so is Algorithm~\ref{alg:s2} and Algorithm~\ref{alg:s4} (See Remark~\ref{rem:alg3} and Remark~\ref{rem:alg4} for the theoretical difference and Example~\ref{exam:1} for the numerical verification). This motivates us to study the convergence of Algorithm~\ref{alg:s3} and  Algorithm~\ref{alg:s4}. In addition, the discussion of the optimal iteration parameters of Algorithm~\ref{alg:s3} and  Algorithm~\ref{alg:s4} is also lack in  \cite{alpa2023}. Hence, the second goal of this paper is to discuss the optimal iteration parameters of Algorithm~\ref{alg:s3} and  Algorithm~\ref{alg:s4} to some extend.

\begin{alg}\label{alg:s3}
Let $A\in \mathbb{R}^{n\times n}$ be nonsingular and $b\in \mathbb{R}^{n}$. Given an initial vector $x^{(0)}\in \mathbb{R}^{n}$, for $k=0,1,2,\cdots$ until the iteration sequence $\left\{x^{(k)}\right\}_{k=0}^\infty$ is convergent, compute
\begin{eqnarray}\label{eq:alg1}
\begin{cases}
y^{(k+1)}=A^{-1}(|x^{(k)}|+b),\\
x^{(k+1)}=x^{(k)}-\lambda E[-M_A y^{(k+1)}+N_A x^{(k)}-(|x^{(k)}|+b)],
\end{cases}
\end{eqnarray}
where $\lambda>0$ is a given parameter and $E\in\mathbb{R}^{n\times n}$ is a given diagonal matrix with positive diagonal components.
\end{alg}

\begin{alg}\label{alg:s4}
Let $A\in \mathbb{R}^{n\times n}$ be nonsingular and $b\in \mathbb{R}^{n}$. Given an initial vector $x^{(0)}\in \mathbb{R}^{n}$, for $k=0,1,2,\cdots$ until the iteration sequence $\left\{x^{(k)}\right\}_{k=0}^\infty$ is convergent, compute
\begin{eqnarray}\label{eq:alg2}
\begin{cases}
y^{(k+1)} = A^{-1}(|x^{(k)}|+b),\\
x^{(k+1)}=x^{(k)} + D_A^{-1} M_A y^{(k+1)} - \lambda E\left[A x^{(k)}-|x^{(k)}|-b\right] -D_A^{-1} M_A x^{(k)},
\end{cases}
\end{eqnarray}
where $\lambda>0$ is a given parameter and $E\in\mathbb{R}^{n\times n}$ is a given diagonal matrix with positive diagonal components.
\end{alg}

The rest of this paper is organized as follows. In Section~\ref{sec:Preliminaries}, we present some lemmas which are useful for our later developments. In Section~\ref{sec:Convergence}, we present the convergence results for Algorithm~\ref{alg:s3} and  Algorithm~\ref{alg:s4}. In Section~\ref{sec:parameter}, we discuss the optimal iteration parameters of Algorithm~\ref{alg:s3} and  Algorithm~\ref{alg:s4}. Our claims are verified by numerical results in Section~\ref{sec:Numerical}. Finally, some concluding remarks are given in Section~\ref{sec:Conclusions}.

\textbf{Notation.} Let~$\mathbb{R}^{n\times n}$ be the set of all~$n\times n$ real matrices and~$\mathbb{R}^n=\mathbb{R}^{n\times 1}$.~$|U|\in\mathbb{R}^{m\times n}$ denote the componentwise absolute value of matrix~$U$.~$I$ denotes the identity
matrix with suitable dimensions.~$\|U\|$ denotes the 2-norm of~$U\in\mathbb{R}^{m\times n}$ which is defined by the formula~$\|U\|=\max\{\|Ux\|:x\in\mathbb{R}^n,\|x\|=1\}$, where~$\|x\|$ is the 2-norm of the vector~$x$.~$\rho(U)$ denotes the spectral radius of~$U$. For $A \in \mathbb{R}^{n\times n}$, $\det (A)$ denotes its determinant.

\section{Preliminaries}\label{sec:Preliminaries}
In this section, we collect some basic results that will be used later.

\begin{lemma}(\cite[Lemma 2.1.]{youn1971})\label{lem:2.1}
Let $p$ and $q$ be real, then both roots of the quadratic equation $x^2-px+q=0$ are less than one in modulus if and only if $|q|<1$ and $|p|<1+q$.
\end{lemma}

\begin{lemma}(e.g., \cite[Theorem~1.10]{saad2003})\label{lem:2.4}
For~$U\in\mathbb{R}^{n\times n}$,~$\lim\limits_{k\rightarrow+\infty} U^k=0$ if and only if~$\rho(U)<1$.
\end{lemma}

\begin{lemma}(e.g., \cite[Theorem~1.11]{saad2003})\label{lem:2.3}
For~$U\in\mathbb{R}^{n\times n}$, the series~$\sum\limits_{k=0}^\infty U^k$ converges if and only if~$\rho(U)<1$ and we have~$\sum\limits_{k=0}^\infty U^k=(I-U)^{-1}$ whenever it converges.
\end{lemma}

\section{Convergence analysises of Algorithm~\ref{alg:s3} and Algorithm~\ref{alg:s4}}\label{sec:Convergence}
In this section, we will analysis the convergence of Algorithm~\ref{alg:s3} and Algorithm~\ref{alg:s4}. For  Algorithm~\ref{alg:s3}, we have the following convergence theorem.

\begin{theorem}\label{thm:1}
Let $A=N_A - M_A$ with $A$ and $N_A$ being nonsingular. If there exist $\lambda >0$ and a positive diagonal matrix $E$ such that
\begin{equation}\label{eq:condition1}
\lambda\|E\|+\|I-\lambda EN_A\|+\lambda \|A^{-1}\|\|EM_A\|<1,
\end{equation}
then the AVE~\eqref{eq:ave} has a unique solution~$x^{*}$ for any $b\in \mathbb{R}^n$ and the sequence~$\{(x^{(k)},y^{(k)})\}^\infty_{k=0}$ generated by~\eqref{eq:alg1} converges to~$(x^{*}, y^{*}=A^{-1}\left(\vert x^{*}\vert+b)\right)$.
\end{theorem}
\begin{proof}
It follows from~\eqref{eq:alg1} that
\begin{eqnarray}\label{eq:meth1'}
\begin{cases}
y^{(k)}=A^{-1}(|x^{(k-1)}|+b),\\
x^{(k)}=x^{(k-1)}-\lambda E[-M_A y^{(k)}+N_A x^{(k-1)}-(|x^{(k-1)}|+b)].
\end{cases}
\end{eqnarray}
Subtracting~\eqref{eq:meth1'} from~\eqref{eq:alg1}, we have
\begin{eqnarray*}
\begin{cases}
y^{(k+1)}-y^{(k)}=A^{-1}(|x^{(k)}|-|x^{(k-1)}|),\\
x^{(k+1)}-x^{(k)}=(I-\lambda EN_A)(x^{(k)}-x^{(k-1)})+\lambda EM_A (y^{(k+1)}-y^{(k)})+\lambda E(|x^{(k)}|-|x^{(k-1)}|).
\end{cases}
\end{eqnarray*}
By considering the 2-norm on both sides of the above equations and applying the inequality $\||x| - |y|\| \le \|x - y\|$, we get
\begin{eqnarray*}
\begin{cases}
\|y^{(k+1)}-y^{(k)}\|\leq \|A^{-1}\|\|x^{(k)}-x^{(k-1)}\|,\\
\|x^{(k+1)}-x^{(k)}\|\leq (\|I-\lambda EN_A\|+\|\lambda E\|)\|x^{(k)}-x^{(k-1)}\|+\|\lambda EM_A\|\|y^{(k+1)}-y^{(k)}\|,
\end{cases}
\end{eqnarray*}
from which we obtain
\begin{eqnarray}\label{eq:1*}\footnotesize
\left[\begin{array}{cc}
                1 & 0\\
-\|\lambda EM_A\| & 1
\end{array}\right]
\left[\begin{array}{c}
 \|y^{(k+1)}-y^{(k)}\| \\
 \|x^{(k+1)}-x^{(k)}\|
\end{array}\right]
\leq
\left[\begin{array}{cc}
0 & \|A^{-1}\|                      \\
0 & \|\lambda E\|+\|I-\lambda EN_A\|
\end{array}\right]
\left[\begin{array}{c}
 \|y^{(k)}-y^{(k-1)}\| \\
 \|x^{(k)}-x^{(k-1)}\|
\end{array}\right].
\end{eqnarray}
Multiplying both sides of~\eqref{eq:1*} from the left by the nonnegative matrix
$P=\left[\begin{array}{cc}
1 & 0\\
\|\lambda EM_A\| & 1
\end{array}\right],$ we have
\begin{eqnarray}\label{eq:W}
\left[\begin{array}{c}
 \|y^{(k+1)}-y^{(k)}\| \\
 \|x^{(k+1)}-x^{(k)}\|
\end{array}\right]
\leq W
\left[\begin{array}{c}
 \|y^{(k)}-y^{(k-1)}\| \\
 \|x^{(k)}-x^{(k-1)}\|
\end{array}\right],
\end{eqnarray}
where
\begin{eqnarray*}
W=
\left[\begin{array}{cc}
0 & \|A^{-1}\|                      \\
0 & \|\lambda E\|+\|I-\lambda EN_A\|+\|A^{-1}\|\|\lambda EM_A\|
\end{array}\right].
\end{eqnarray*}
For each $m \geq 1$, if $\rho(W)<1$, it follows from~\eqref{eq:W}, Lemma~\ref{lem:2.3} and Lemma~\ref{lem:2.4} that
\begin{align*}
\left[\begin{array}{c}
 \|y^{(k+m)}-y^{(k)}\| \\
 \|x^{(k+m)}-x^{(k)}\|
\end{array}\right]&=
\left[\begin{array}{c}
 \|\sum_{j=0}^{m-1}(y^{(k+j+1)}-y^{(k+j)})\| \\
 \|\sum_{j=0}^{m-1}(x^{(k+j+1)}-x^{(k+j)})\|
\end{array}\right]
\leq
\sum_{j=0}^{m-1}
\left[\begin{array}{c}
 \|y^{(k+j+1)}-y^{(k+j)}\| \\
 \|x^{(k+j+1)}-x^{(k+j)}\|
\end{array}\right]\nonumber\\
&\leq \sum_{j=0}^{\infty}W^{j+1}
\left[\begin{array}{c}
 \|y^{(k)}-y^{(k-1)}\| \\
 \|x^{(k)}-x^{(k-1)}\|
\end{array}\right]
=(I-W)^{-1}W
\left[\begin{array}{c}
 \|y^{(k)}-y^{(k-1)}\| \\
 \|x^{(k)}-x^{(k-1)}\|
\end{array}\right]\nonumber\\
&\leq (I-W)^{-1}W^k
\left[\begin{array}{c}
 \|y^{(1)}-y^{(0)}\| \\
 \|x^{(1)}-x^{(0)}\|
\end{array}\right]
\rightarrow
\left[\begin{array}{c}
0\\
0
\end{array}\right]~~(\text{as}\quad k\rightarrow \infty).
\end{align*}
Therefore,~$\{y^{(k)}\}_{k=0}^{\infty}$ and~$\{x^{(k)}\}_{k=0}^{\infty}$ are Cauchy sequences and hence they are convergent in $\mathbb{R}^n$. Let $\lim\limits_{k\rightarrow\infty} y^{(k)} =y^{*}$  and $\lim\limits_{k\rightarrow\infty} x^{(k)} =x^{*}$. Then it follows from~\eqref{eq:alg1} that
\begin{equation*}
\begin{cases}
y^{*}=A^{-1}(|x^{*}|+b),\\
x^{*}=x^{*}-\lambda E[-M_A y^{*}+N_A x^{*}-(|x^{*}|+b)],
\end{cases}
\end{equation*}
from which and the nonsingularity of $N_A$  we have
\begin{eqnarray*}
\begin{cases}
Ax^{*}-|x^{*}|-b=0,\\
x^{*}=y^{*}.
\end{cases}
\end{eqnarray*}
In other words, $x^{*}$ is a solution to AVE~\eqref{eq:ave}.

Next, we will prove that $\rho(W)<1$ if~\eqref{eq:condition1} holds.
Suppose that~$\tau$ is an eigenvalue of~$W$, then
\begin{eqnarray*}
\det (\tau I-W)=\det\left(
\left[\begin{array}{cc}
 \tau  &  -\|A^{-1}\| \\
   0   &  \tau-(\lambda\|E\|+\|I-\lambda EN_A\|+\lambda \|A^{-1}\|\|EM_A\|)
\end{array}\right]\right)=0,
\end{eqnarray*}
from which we have
$$\tau^2-\tau(\lambda\|E\|+\|I-\lambda EN_A\|+\lambda \|A^{-1}\|\|EM_A\|)=0.$$
It follows from Lemma~\ref{lem:2.1}  that $|\tau|<1$ (i.e., $\rho(W)<1$) if and only if
$$\lambda\|E\|+\|I-\lambda EN_A\|+\lambda \|A^{-1}\|\|EM_A\|<1,$$
which is~\eqref{eq:condition1}.

Finally, if \eqref{eq:condition1} holds, we can prove the unique solvability of AVE~\eqref{eq:ave}. In contrast, suppose that~$x^{*}$ and~$\bar{x}^{*}$ are two different solutions of the AVE~\eqref{eq:ave}, then we have
\begin{numcases}{}
\|y^{*}-\bar{y}^{*}\| \leq \|A^{-1}\|\|x^{*}-\bar{x}^{*}\|,\label{eq:xbxa}\\
\|x^{*}-\bar{x}^{*}\| \leq (\|I-\lambda EN_A\|+\|\lambda E\|) \|x^{*}-\bar{x}^{*}\|+\|\lambda EM_A\|\|y^{*}-\bar{y}^{*}\|,\label{eq:xbxb}
\end{numcases}
where $y^{*}=x^{*} = A^{-1}(|x^*| + b)$ and $\bar{y}^{*}=\bar{x}^{*} = A^{-1}(|\bar{x}^*| + b)$. Note that it can be deduced from~\eqref{eq:condition1} that~$\|I-\lambda EN_A\|+\|\lambda E\| < 1$, it follows from~\eqref{eq:xbxb} that
\begin{equation}\label{eq:xy}
\|x^{*}-\bar{x}^{*}\| \leq \frac{\|\lambda EM_A\|}{1-\|I-\lambda EN_A\|-\|\lambda E\|} \|y^{*}-\bar{y}^{*}\|.
\end{equation}
From~\eqref{eq:condition1}, we  also have
\begin{equation}\label{eq:condition1.1}
\frac{\|A^{-1}\|\|\lambda EM_A\|}{1-\|I-\lambda EN_A\|-\|\lambda E\|}<1.
\end{equation}
It follows from~\eqref{eq:xbxa}, \eqref{eq:xy} and \eqref{eq:condition1.1} that
\begin{align*}
\|y^{*}-\bar{y}^{*}\| &\leq \|A^{-1}\|\|x^{*}-\bar{x}^{*}\|\nonumber\\
                  &\leq \frac{\|A^{-1}\|\|\lambda EM_A\|}{1-\|I-\lambda EN_A\|-\|\lambda E\|} \|y^{*}-\bar{y}^{*}\|\nonumber\\
                  &< \|y^{*}-\bar{y}^{*}\|,
\end{align*}
which will lead to a contradiction whenever~$y^{*}\neq\bar{y}^{*}$ (since~$x^{*}\neq\bar{x}^{*}$). Hence, we have~$x^{*}=\bar{x}^{*}$. The proof is completed.
\end{proof}

\begin{remark}\label{rem:alg3}
If $n = 1$, it can be checked that \eqref{eq:condition1} is the same with \eqref{ie:condition1}. However, when $n \ge 2$, the following examples shows that \eqref{eq:condition1} does not imply \eqref{ie:condition1} and vice versa. Hence, when $n\ge 2$, \eqref{eq:condition1} can not theoretically guarantee the convergence of Algorithm \ref{alg:s1} as well as \eqref{ie:condition1} can not theoretically guarantee the convergence of Algorithm \ref{alg:s3}.

For
\begin{equation*}
A=\begin{bmatrix}
        6 & -0.5\\ -0.9 & 1.5
        \end{bmatrix},
\end{equation*}
$\lambda = 0.9$ and $E = D_A^{-1}$, we have $\rho(R^{-1}S) \approx 0.8291$ and $ \lambda\|E\|+\|I-\lambda EN_A\|+\lambda \|A^{-1}\|\|EM_A\| \approx 1.5254.$ For
\begin{equation*}
A=\begin{bmatrix}
        5 & -3.6\\ -1.1 & 8
        \end{bmatrix},
\end{equation*}
$\lambda = 0.9$ and $E = D_A^{-1}$, we have
$\rho(R^{-1}S) \approx 1.0318$ and $\lambda\|E\|+\|I-\lambda EN_A\|+\lambda \|A^{-1}\|\|EM_A\| \approx 0.9754.$
\end{remark}

For  Algorithm~\ref{alg:s4}, we have the following convergence theorem.

\begin{theorem}\label{thm:2}
Let $A=N_A-M_A$ with $A$ and $N_A$ being nonsingular. If there exist $\lambda >0$ and a positive diagonal matrix $E$ such that $D_A^{-1}M_A + \lambda E A$ is nonsingular and
\begin{equation}\label{eq:condition2}
\|\lambda E\|+\|I-\lambda E A-D_A^{-1} M_A\|+\|D_A^{-1} M_A\|\|A^{-1}\|<1,
\end{equation}
then the AVE~\eqref{eq:ave} has a unique solution~$x^*$ for any $b\in \mathbb{R}^n$ and the sequence~$\{(x^{(k)},y^{(k)})\}^\infty_{k=0}$ generated by~\eqref{eq:alg2} converges to~$(x^{*}, y^{*}=A^{-1}(\vert x^{*}\vert+b))$.
\end{theorem}
\begin{proof}
The proof is similar to that of Theorem~\ref{thm:1} and hence we omit it here to save space.
\end{proof}

\begin{remark}\label{rem:alg4}
When $n = 1$, it can be proven that  \eqref{eq:condition2} implies \eqref{ie:condition2} while the reverse is not true. For instance, let $\lambda E = 1$ and $A = \frac{3}{2}$. When $n \ge 2$, the following examples shows that \eqref{eq:condition2} does not imply \eqref{ie:condition2} and vice versa. Hence, when $n\ge 2$, \eqref{eq:condition2} can not theoretically guarantee the convergence of Algorithm \ref{alg:s2} as well as \eqref{ie:condition2} can not theoretically guarantee the convergence of Algorithm \ref{alg:s4}.

For
\begin{equation*}
A= {\rm tridiag} (-3,9,-3)\in\mathbb{R}^{n \times n}\quad \text{with}\quad n=500,
\end{equation*}
$\lambda = 0.9$ and $E = D_A^{-1}$, we have $\rho(G^{-1}J)\approx 1.2258$ and $\|\lambda E\|+\|I-\lambda E A-D_A^{-1} M_A\|+\|D_A^{-1} M_A\|\|A^{-1}\|\approx 0.9964.$ For
\begin{equation*}
A=\begin{bmatrix}
        5 & -3\\ -3 & 7
        \end{bmatrix},
\end{equation*}
$\lambda = 0.9$ and $E = D_A^{-1}$, we have $\rho(G^{-1}J)\approx 0.4600$ and $\|\lambda E\|+\|I-\lambda E A-D_A^{-1} M_A\|+\|D_A^{-1} M_A\|\|A^{-1}\|\approx 1.0318.$
\end{remark}

\section{The optimal iteration parameters for Algorithm~\ref{alg:s3} and Algorithm~\ref{alg:s4}}\label{sec:parameter}
The performance of Algorithm~\ref{alg:s3} and  Algorithm~\ref{alg:s4} are dependent on the choices of $\lambda$ and $E$. It is an important problem to determine the optimal choices of $\lambda$ and $E$. By an optimal choice, we mean it is the choice such that the corresponding method gets the fastest convergence rate. Usually, it seems not to be an easy task to find the optimal choices of $\lambda$ and $E$ for Algorithm~\ref{alg:s3} and  Algorithm~\ref{alg:s4}, while it remains significance to find a somewhat optimal one. Hence, in this section, we will determine the optimal values of $\lambda$ for Algorithm~\ref{alg:s3} and  Algorithm~\ref{alg:s4} for certain $E$.

\subsection{The optimal iteration parameter for Algorithm~\ref{alg:s3}}
We first consider the optimal iteration parameter of Algorithm~\ref{alg:s3}. It follows from the proof of Theorem~\ref{thm:1} that the smaller value of~$\rho(W)$ is, the faster the Algorithm~\ref{alg:s3} will converge later on. Let
\begin{equation*}
g(\lambda,E)=\lambda\|E\|+\|I-\lambda EN_A\|+\lambda \|EM_A\|\|A^{-1}\|.
\end{equation*}
Since $W\ge 0$, the question turns to find $\lambda$ and $E$ such that $g(\lambda,E)$ is minimized. However, this is general not an easy task since $\lambda$ and $E$ are coupled with each other. In the following, by freezing $E$, we discuss the optimal value of $\lambda$.

\begin{theorem}\label{thm:alg3}
Let $A$ be a nonsingular matrix and $A = N_A - M_A$ with $N_A$ being nonsingular.
If~$E=N_A^{-1}$ and $\lambda >0$ is the parameter such that \eqref{eq:condition1} holds, then the optimal iteration parameter which minimizes~$\rho(W)$ is~$\lambda=1$.
\end{theorem}
\begin{proof}
When~$E=N_A^{-1}$, we have
\begin{align}\nonumber
g(\lambda,E)&=|1-\lambda|+\lambda \|N_A^{-1}\|+\lambda \|N_A^{-1}(N_A-A)\|\|A^{-1}\|\\\nonumber
              &=|1-\lambda|+\lambda \|N_A^{-1}\|+\lambda \|I-N_A^{-1}A\|\|A^{-1}\|.
\end{align}

When~$0<\lambda \leq 1$, let~$\mu_1=\|N_A^{-1}\|+\|I-N_A^{-1}A\|\|A^{-1}\|-1 $, then we have
\begin{align*}
g(\lambda,E)&=1-\lambda+\lambda \|N_A^{-1}\|+\lambda \|I-N_A^{-1}A\|\|A^{-1}\|, \nonumber\\
&=1+\lambda(\|N_A^{-1}\|+ \|I-N_A^{-1}A\|\|A^{-1}\|-1)\nonumber\\
&=1+\lambda \mu_1.
\end{align*}
It is easy to calculate that~$\frac{\partial g}{\partial \lambda}=\mu_1.$ In addition, it follows from \eqref{eq:condition1} that $1+\lambda \mu_1<1$, which combines $\lambda>0$ implies $\mu_1<0$. Hence, $g(\cdot, E)$ obtains its minimum value at $\lambda = 1$ whenever $\lambda \in (0,1]$.

When~$\lambda\ge 1$, let~$\mu_2=\|N_A^{-1}\|+\|I-N_A^{-1}A\|\|A^{-1}\|+1\in (1,2)$, then we have
$$g(\lambda,E)=\lambda \mu_2-1,$$
from which $\frac{\partial g}{\partial \lambda}=\mu_2 >1$. Hence, $g(\cdot, E)$ obtains its minimum value at $\lambda = 1$ whenever $\lambda \in \left[1, \frac{2}{\mu_2}\right)$.
\end{proof}

\subsection{The optimal iteration parameter for Algorithm \ref{alg:s4}}
Next, we discuss the optimal iteration parameter of Algorithm~\ref{alg:s4}. Let
\begin{equation*}
g(\lambda,E)=\|I-\lambda EA-D_A^{-1}M_A\|+\|\lambda E\|+\|D_A^{-1}M_A\| \|A^{-1}\|,
\end{equation*}
by fixing $E$, the goal is to minimize $g$ with respect to $\lambda$.

\begin{theorem}\label{thm:alg4}
Let~$E=D_A^{-1}$ and $N_A=D_A$ (i.e. $U_A = 0$). If $\lambda >0$ is the parameter such that the conditions of Theorem~\ref{thm:2} hold, then the optimal iteration parameter for Algorithm~\ref{alg:s4} is~$\lambda=1$.
\end{theorem}
\begin{proof}
Based on the assumptions, we have
\begin{equation}\label{eq:g1'}
g(\lambda,E)
=|1-\lambda|\|D_A^{-1}A\|+\lambda\|D_A^{-1}\|+\|D_A^{-1}L_A\|\|A^{-1}\|.
\end{equation}

When $0<\lambda \leq 1$, denote $\alpha_1 =\|D_A^{-1}\|-\|D_A^{-1}A\|, \beta_1=\|D_A^{-1}A\|+\|D_A^{-1}L_A\|\|A^{-1}\|$, it follows from \eqref{eq:g1'} that
\begin{align*}
 g(\lambda,E)&= \|D_A^{-1}A\|+\|D_A^{-1}L_A\|\|A^{-1}\| + \lambda(\|D_A^{-1}\|-\|D_A^{-1}A\|) \nonumber\\
               &=\alpha_1\lambda+\beta_1.
\end{align*}
It is easy to calculate that $\frac{\partial g}{\partial \lambda}=\alpha_1.$ In the following, we argue that $\alpha_1< 0$. Otherwise, if $\alpha_1\ge 0$, then $\|D_A^{-1}\|\ge \|D_A^{-1}A\|$, from which we have
$$
\|D_A^{-1}\|\|A^{-1}\|\ge \|D_A^{-1}A\|\|A^{-1}\| \geq\|D_A^{-1}AA^{-1}\|
=\|D_A^{-1}\|.
$$
Thus we get $\|A^{-1}\|\ge 1$. Then
\begin{align*}
\|I-\lambda EA-D_A^{-1}M_A\|&+\|\lambda E\|+\|D_A^{-1}M_A\|\|A^{-1}\|\\
& \ge 1- \lambda\|D_A^{-1}A\| - \|D_A^{-1}L_A\| + \lambda \|D_A^{-1}\| + \|D_A^{-1}L_A\| \\
&\ge 1 -  \lambda\|D_A^{-1}\|+ \lambda \|D_A^{-1}\|\\
&= 1,
\end{align*}
which contradicts with \eqref{eq:condition2}. Hence, $\frac{\partial g}{\partial \lambda} < 0$ when $\lambda\in \left(\max\left\{0, \frac{1-\beta_1}{\alpha_1}\right\},1\right]$. Namely, $g$ is strictly monotonously decreasing in terms of $\lambda$ in $\left(\max\left\{0, \frac{1-\beta_1}{\alpha_1}\right\},1\right]$ and the optimal iteration parameter is~$\lambda=1$.

When~$\lambda \ge 1$, denote $\alpha_2 =\|D_A^{-1}\|+\|D_A^{-1}A\|, \beta_2=\|D_A^{-1}L_A\|\|A^{-1}\| -\|D_A^{-1}A\|$, it follows from $\lambda\ge 1$ and \eqref{eq:condition2} that $\frac{1-\beta_2}{\alpha_2} >1$. By~\eqref{eq:g1'}, we have
\begin{align*}
 g(\lambda,E)&= \|D_A^{-1}L_A\|\|A^{-1}\| - \|D_A^{-1}A\|+\lambda(\|D_A^{-1}\|+\|D_A^{-1}A\|) \nonumber\\
               &=\alpha_2\lambda+\beta_2,
\end{align*}
then $\frac{\partial g}{\partial \lambda} = \alpha_2 >0.$ That is, $g$ is strictly monotonously increasing with respect to $\lambda$ in $\left[1, \frac{1-\beta_2}{\alpha_2}\right)$ and the optimal iteration parameter is $\lambda=1$.

To sum up, the optimal iteration parameter is~$\lambda=1$.
\end{proof}

\section{Numerical results}\label{sec:Numerical}
In this section, we will present two numerical examples to demonstrate our claims in the previous sections. For algorithms will be tested, i.e., Algorithms~\ref{alg:s1}--\ref{alg:s4}. In the numerical results, we will report ``IT'' (the number of iterations), ``Time'' (the elapsed CPU time in seconds) and ``RES'' which is defined by
$$
{\rm RES} =\frac{\|b+|x^{(k)}|-Ax^{(k)}\|}{\|b\|}.
$$
All tests are started from initial zero vector and terminated if the current iteration satisfies ${\rm RES} \le 10^{-6}$ or the number of prescribed maximal iteration step $k_{\max} = 500$ is exceeded. In our computations, all runs are implemented in MATLAB (version 9.10 (R2021a)) on a personal computer with IntelCore(TM) i7 CPU 2.60 GHz, 16.0GB memory.

\begin{example}\label{exam:1}
Suppose that $A=M+\mu I\in\mathbb{R}^{n\times n}$ and $b = Ax^*-|x^*|\in\mathbb{R}^n$ with~$M={\rm tridiag}(-1.5I,S,-0.5I)\in\mathbb{R}^{n\times n}$ and $x^*=(-1,1,-1,1,\cdots, -1,1)^\top\in\mathbb{R}^n$, where~$S={\rm tridiag}(-1.5,8,-0.5)\in\mathbb{R}^{m\times m}$.

In this example, we first compare the performances of Algorithms~\ref{alg:s1}--\ref{alg:s4}. Numerical results are report in Table~\ref{table1} and Figure~\ref{fig1}, from which we can find that Algorithm~\ref{alg:s1} and Algorithm~\ref{alg:s3} have different numerical performances, and the same goes to Algorithm~\ref{alg:s2} and Algorithm~\ref{alg:s4}. Next, we will test the optimal iteration parameter of Algorithm~\ref{alg:s3}. For this purpose, we fix $E = N_A^{-1}$ and numerical results are shown in Figure~\ref{alg3opt}, which intuitively confirms our conclusion in Theorem~\ref{thm:alg3}.

\begin{table}[h]
\centering
\caption{Numerical results for Examples~\ref{exam:1} with $E = D_A^{-1}$, $\lambda=0.5$ and $\mu=4$.}\label{table1}
\begin{tabular}{lllllllll}\hline
&n=400 &Algorithm~\ref{alg:s1}    &Algorithm~\ref{alg:s2}    &Algorithm~\ref{alg:s3}    & Algorithm~\ref{alg:s4}\\\cline{3-6}
&IT       &$26$            &$23$            &$23$             &$29$\\\cline{3-6}
&Time    &$0.0901 $      &$0.0271$       &$0.0445 $       &$0.0581$\\\cline{3-6}
&RES     &$8.3447\times 10^{-7}$  &$6.5155\times 10^{-7}$  &$6.3796 \times 10^{-7}$  &$7.0597\times 10^{-7}$ \\\hline
\end{tabular}
\end{table}

\begin{figure}[h]
\centering  
\subfigure{
\includegraphics[width=0.45\linewidth]{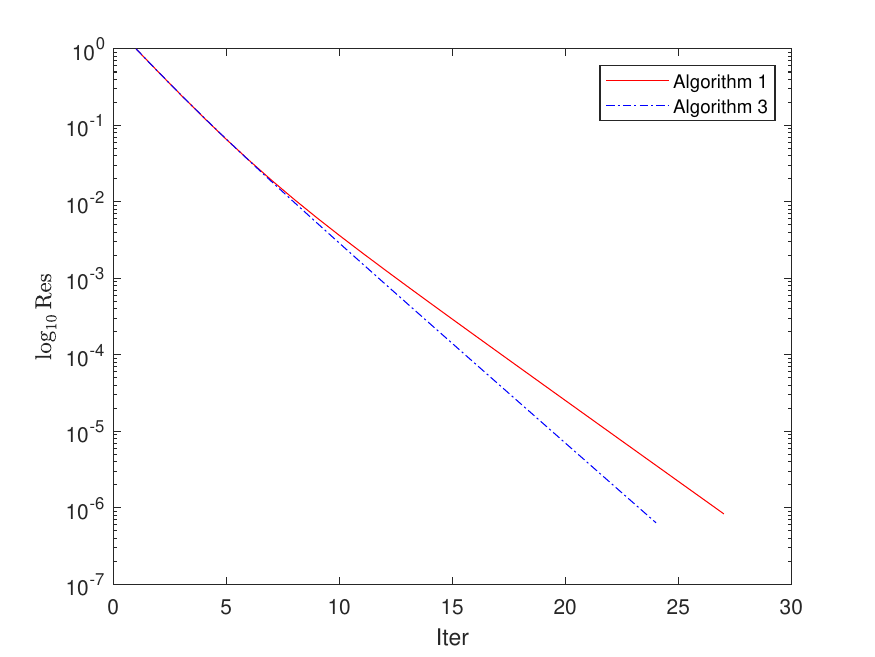}}
\subfigure{
\includegraphics[width=0.45\linewidth]{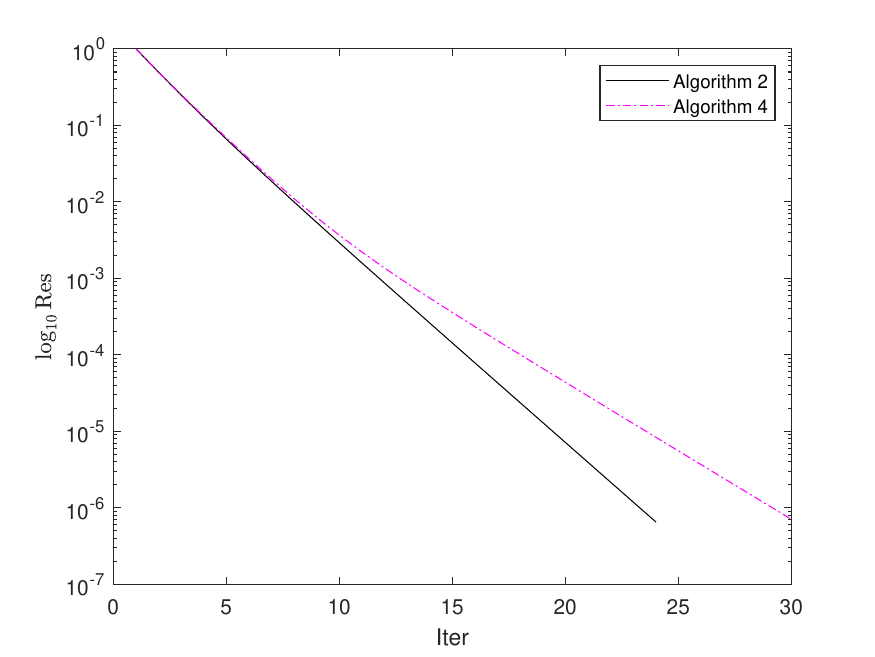}}
\caption{Iterative curves for Example~\ref{exam:1} with~$\lambda=0.5$ and~$\mu=4$.}
\label{fig1}
\end{figure}

\begin{figure}[h]
  \centering
  \includegraphics[width=0.7\linewidth]{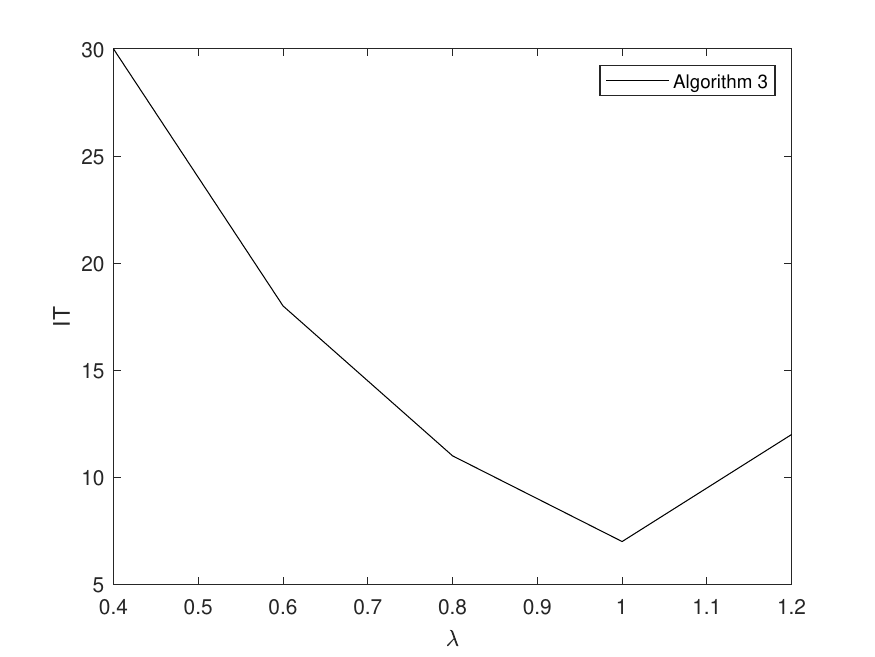}
  \caption{The trend diagram of IT versus $\lambda$ for Algorithm~\ref{alg:s3} in Example~\ref{exam:1}.}\label{alg3opt}
\end{figure}
\end{example}

In the analysis of Theorem~\ref{thm:alg4}, $A$ is a lower triangular matrix. We develop the following example to verify the result in Theorem~\ref{thm:alg4}. Since lower triangular matrix is also applicable to the analysis of Theorem~\ref{thm:alg3}, Algorithm~\ref{alg:s3} is also tested.

\begin{example}\label{exam:2}
Suppose the matrix~$A\in\mathbb{R}^{n\times n}$ is presented by
\begin{eqnarray*}
A=
\begin{cases}
S,\quad \text{if} \quad j=i;\\
-I,\quad \text{if}\quad j=i-1 (i=2,\cdots,n);\\
0, \quad \text{otherwise},
\end{cases}
\end{eqnarray*}
where~$S={\rm tridiag}(-1,\theta,0)\in\mathbb{R}^{m\times m}$, $I\in\mathbb{R}^{m\times m}$ represents the unit matrix. Let~$b=Ax^*-|x^*|\in\mathbb{R}^n$ with $x^*=(-1,1,-1,1,\cdots, -1,1)^\top\in\mathbb{R}^n$. The trend diagrams of IT versus $\lambda$ for Algorithms~\ref{alg:s3} and~\ref{alg:s4} are reported in Figure~\ref{alg34opt}, from which we can find that the results in Section~\ref{sec:parameter} are verified.

\begin{figure}[h]
  \centering
  \includegraphics[width=0.7\linewidth]{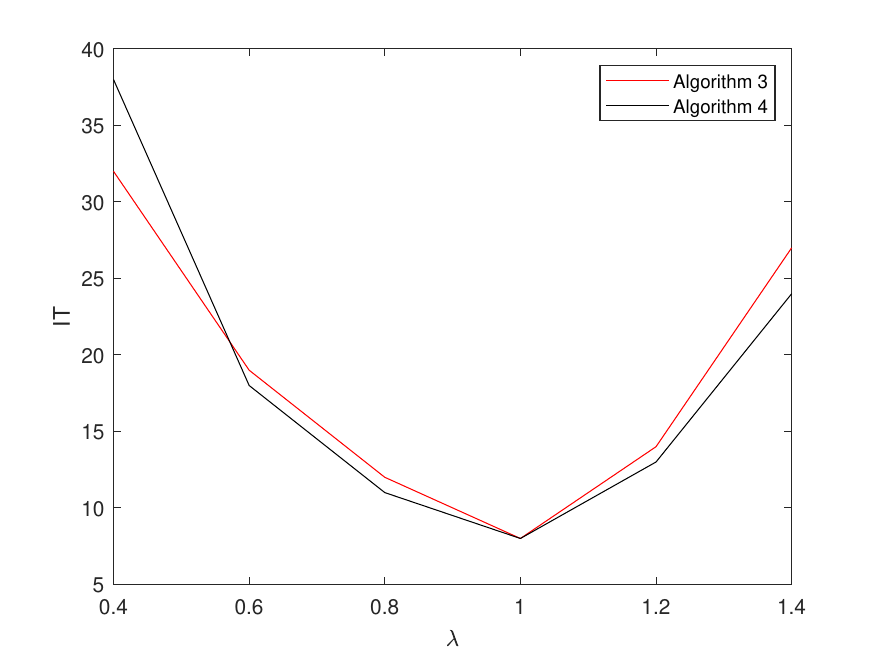}
  \caption{The trend diagrams of IT versus $\lambda$ for Algorithms~\ref{alg:s3} and~\ref{alg:s4} in Example~\ref{exam:2}.}\label{alg34opt}
\end{figure}
\end{example}

\section{Conclusions}\label{sec:Conclusions}
In this paper, we establish the convergence results of the two iterative methods proposed in~\cite{alpa2023}, and discuss the optimal iteration parameters under certain special cases.
Two numerical examples are given to demonstrate our claims. How to determine the optimal iteration parameters of the algorithms in a more general case is worthy of further study.

%

\section*{Disclosure statement}

No potential conflict of interest was reported by the authors.

\section*{Funding}

Cairong Chen was supported by the Natural Science Foundation of Fujian Province (2021J01661).

%


\end{document}